\documentclass[12pt]{article}

\usepackage[T1]{fontenc}
\usepackage{lmodern}
\usepackage{microtype}
\usepackage{amsmath,amsthm,amsfonts,amssymb}
\usepackage{mathtools}
\usepackage{enumitem}
\usepackage{hyperref}
\usepackage{geometry}
\hypersetup{colorlinks=true,linkcolor=blue,citecolor=blue,urlcolor=blue}

\newtheorem{theorem}{Theorem}[section]
\newtheorem{lemma}[theorem]{Lemma}
\newtheorem{proposition}[theorem]{Proposition}
\newtheorem{corollary}[theorem]{Corollary}
\theoremstyle{definition}
\newtheorem{definition}[theorem]{Definition}
\newtheorem{remark}[theorem]{Remark}
\newtheorem{example}[theorem]{Example}

\newcommand{\id}{\mathrm{id}}

\title{Proximal Relations and Maximal Equicontinuous Factors
for Non-autonomous Dynamical Systems}
\author{%
  Saksham Malik\thanks{Department of Mathematics, University of Delhi, Delhi-110007, India.
    \newline\href{mailto:sakshamalik@gmail.com}{sakshamalik@gmail.com}}%
  \and
  Mohammad Salman\thanks{Department of Mathematics, Shyama Prasad Mukherji College for Women, University of Delhi, Delhi-110026, India.
    \newline\href{mailto:salman25july@gmail.com}{salman25july@gmail.com}}%
  \and
  Ruchi Das\thanks{Department of Mathematics, University of Delhi, Delhi-110007, India.
    \newline\href{mailto:rdasmsu@gmail.com}{rdasmsu@gmail.com} (Corresponding Author)} }

\date{ }
\begin{document}

\maketitle

\begin{abstract}
This paper studies proximal-type relations and maximal equicontinuous factors for non-autonomous dynamical systems on compact metric spaces. Two classes of systems are considered: systems generated by uniformly convergent sequences of maps and periodic systems. For uniformly convergent generators, collective convergence and uniform rigidity imply that the regional proximal relation for the non-autonomous system is equal to the regional proximal relation for the limit autonomous system and if the limit system is equicontinuous, then the proximality relation for the non-autonomous system is equal to the proximal relation for the autonomous limit system. In the same setting, the kernel of the maximal equicontinuous factor is identified as the smallest closed equivalence relation containing the autonomous equicontinuous structure relation of the corresponding autonomous system and preserved by all generators. For periodic systems, regional proximality and Banach proximality are determined by the period map, yielding a criterion in terms of mean equicontinuity and the relation of sensitivity in the mean. Examples show that the hypotheses used in these results are necessary.
\end{abstract}


\section{Introduction}

Maximal equicontinuous factors and proximal-type relations are central objects in topological dynamics. In the autonomous setting, the maximal equicontinuous factor gives the canonical equicontinuous quotient, and its kernel is closely connected with the equicontinuous structure relation and regional proximality. This connection goes back to the work of Veech \cite{Veech1968}, Ellis and Keynes \cite{EllisKeynes1971}, and Auslander, Ellis, and Ellis \cite{AuslanderEllisEllis1995}. Related developments in mean dynamics connect regional proximality, Banach proximality, mean equicontinuity, and the relation $Q_{me}$ \cite{LiTuYe2015,QiuZhao2020}.

For non-autonomous dynamical systems, the corresponding structure is more delicate. The dynamics is generated by a sequence of maps, so orbit relations are not governed by the iterates of a single transformation. Since the work of Kolyada and Snoha \cite{KolyadaSnoha1996}, several aspects of non-autonomous dynamics have been studied, including transitivity, equicontinuity, proximality, almost periodicity, periodicity, and entropy \cite{AcostaSanchis2019,KloedenRasmussen2011,SharmaRaghav2018General,YadavSharma2025, YANG2026130635}. Systems generated by uniformly convergent sequences of maps and periodic systems have also been investigated in \cite{malik2026auslander,MalikSalmanDas2026,RaghavSharma2017}. The present paper studies these systems from the viewpoint of relation theory and maximal equicontinuous factors.

The guiding question is when the proximal-type relations and equicontinuous quotient structure of a non-autonomous system can be described through a naturally associated autonomous system. This question is subtle even when such an autonomous model is available. Uniform convergence of the generators, for example, controls only finite-time behaviour and does not by itself determine regional proximality or factor kernels. These objects depend on orbit relations over unbounded time scales, and the non-autonomous system may retain effects that are invisible in the limiting map.

The paper treats two tractable classes of non-autonomous systems. For systems generated by uniformly convergent sequences of maps, collective convergence and uniform rigidity provide the additional control needed to relate the non-autonomous relations to the relations of the limiting map. Under these hypotheses, the regional proximal relation of the non-autonomous system coincides with that of the limit system:
\[
RP_{f_{1,\infty}}=RP_g.
\]
When the limit system is equicontinuous, the corresponding proximal relations are also recovered. Examples show that the conclusion may fail if either collective convergence or uniform rigidity is omitted.

The same uniformly convergent setting also yields a description of maximal equicontinuous factors. Working with factor maps that intertwine each generator individually, the paper proves existence of the maximal equicontinuous factor for non-autonomous systems on compact metric spaces. Under collective convergence and uniform rigidity, its kernel is identified as the smallest closed equivalence relation containing the autonomous equicontinuous structure relation $S_g$ of the limit map and preserved by all generators. Thus the non-autonomous maximal equicontinuous factor is obtained from the classical autonomous structure relation by imposing the invariance required by the full non-autonomous sequence.

Periodic systems are also explored. In this class, regional proximality and Banach proximality reduce to the corresponding relations of the period map, and the resulting relation theory connects naturally with mean equicontinuity and the condition $Q_{me}(g)=\Delta_X$. These results complement the uniformly convergent case by showing how the structure simplifies when the non-autonomous dynamics has a periodic form.

Examples and counterexamples are included throughout to mark the scope of the hypotheses. They show that non-autonomous proximal relations and equicontinuous factor kernels may behave substantially differently from their autonomous counterparts, even in systems that are close to autonomous ones. The results therefore identify concrete settings in which the classical relation-theoretic and factor-theoretic picture can be recovered.

The paper is organized as follows. Section~\ref{sec:prelim} fixes notation and recalls the definitions used throughout. Section~\ref{sec:relations} studies proximality, regional proximality, Banach proximality, and mean relations in the uniformly convergent and periodic settings. Section~\ref{sec:mef} develops maximal equicontinuous factors for non-autonomous systems and identifies their kernels under collective convergence and uniform rigidity.

\section{Preliminaries}\label{sec:prelim}

In this section we fix notation and record the notions used throughout the paper. Let $\mathbb N=\{1,2,\dots\}$ and $\mathbb N_0=\mathbb N\cup\{0\}$. For a finite set $E$, the notation $|E|$ denotes its cardinality; in particular, if $I\subseteq\mathbb N_0$ is a finite interval, then $|I|$ denotes the number of integer points in $I$. All spaces appearing below are compact metric spaces unless explicitly stated otherwise. Thus $(X,d)$ denotes a compact metric space, $\Delta_X:=\{(x,x):x\in X\}$ is the diagonal of $X\times X$, and for $x\in X$ and $\varepsilon>0$ we write $B(x,\varepsilon):=\{y\in X:d(x,y)<\varepsilon\}$.

A \emph{non-autonomous dynamical system} on $X$ is a pair $(X,f_{1,\infty})$, where $f_{1,\infty}=(f_n)_{n\ge 1}$ is a sequence of continuous self-maps of $X$. For integers $n\ge 1$ and $i,r\ge 1$ we write
\[
 f_1^n:=f_n\circ\cdots\circ f_1,
 \qquad
 f_i^r:=f_{i+r-1}\circ\cdots\circ f_i,
 \qquad 
 f_i^0=\id_X\qquad (i\ge 1).
\]

Thus $f_1^n$ is the initial composition of length $n$, while $f_i^r$ is the $r$-step tail composition beginning at time $i$. When $f_n=f$ for every $n\ge 1$, one recovers the classical autonomous system $(X,f)$. For continuous maps $T,S\colon X\to X$ we use the uniform metric
\[
 \|T-S\|_\infty:=\sup_{x\in X} d\bigl(T(x),S(x)\bigr).
\]
The system $(X,f_{1,\infty})$ is \emph{equicontinuous} if the family $\{f_1^n:n\ge 0\}$ is equicontinuous on $X$. We say that $(X,f_{1,\infty})$ is \emph{periodic of period $p$} if $p\in\mathbb N$ is the least positive integer such that $f_{n+p}=f_n $ for $n \ge 1$. In that case the associated \emph{period map} is $g:=f_1^p=f_p\circ\cdots\circ f_1$.

A relation $R\subseteq X\times X$ is said to be \emph{preserved} by a continuous map $T\colon X\to X$ if $(T\times T)(R)\subseteq R$.

\begin{definition}\label{def:collective}\cite{RaghavSharma2017}
Let $g\colon X\to X$ be continuous with $f_n\to g$ uniformly. We say that $\{f_i^r\}_{i,r\ge 1}$ \emph{converge collectively} to $\{g^r\}_{r\ge 1}$ if for every $\varepsilon>0$ there exists $N\in\mathbb N$ such that
\[
 \|f_i^r-g^r\|_\infty<\varepsilon
 \quad\text{for all } i\ge N \text{ and } r\ge 1.
\]
\end{definition}

\begin{definition}\label{def:eq-ur-periodic}\cite{BenRejeb2021}
The system $(X,f_{1,\infty})$ is \emph{uniformly rigid} if there exists a sequence $n_k\to\infty$ such that
\[
 \|f_1^{n_k}-\id_X\|_\infty \longrightarrow 0.
\]
\end{definition}

\begin{definition}\label{def:enveloping-semigroup}\cite{Ellis1969}
Let \(g\colon X\to X\) be continuous. The \emph{enveloping semigroup} of the autonomous system \((X,g)\) is $E(X,g):=\overline{\{g^m:m\in\mathbb N_0\}}\subseteq X^X$, where \(X^X\) is endowed with the product topology, equivalently the topology of pointwise convergence. Thus \(q\in E(X,g)\) if and only if there exists a net \((g^{m_\alpha})\) such that $g^{m_\alpha}(z)\longrightarrow q(z)$ for every  $z\in X$.
Elements of \(E(X,g)\) need not be continuous in general.
\end{definition}

\begin{definition}\label{def:factor}
A \emph{factor map} $\pi\colon (X,f_{1,\infty})\longrightarrow (Y,h_{1,\infty})$ between non-autonomous systems is a continuous surjection satisfying $\pi\circ f_n=h_n\circ \pi$ for all $n \in \mathbb N$.

If $R\subseteq X\times X$ is a closed equivalence relation, then $\pi_R\colon X\to X/R$ denotes the quotient map.
\end{definition}

\begin{definition}\label{def:symbolic-shift}\cite{LindMarcus1995}
Let $\mathcal A$ be a finite alphabet. The \emph{one-sided full shift} over $\mathcal A$ is the compact space $\mathcal A^{\mathbb N}=\{x=(x_j)_{j\ge1}:x_j\in\mathcal A\}$ with the product topology. The \emph{shift map} $\sigma\colon \mathcal A^{\mathbb N}\to\mathcal A^{\mathbb N}$ is defined by $\sigma((x_j)_{j\ge1})=(x_{j+1})_{j\ge1}$.

A \emph{finite word} over $\mathcal A$ is an element $w=w_1\cdots w_\ell\in\mathcal A^\ell$ for some $\ell\in\mathbb N$. If $u$ and $v$ are finite words, then $uv$ denotes their concatenation. If $u$ is a finite word and $x\in\mathcal A^{\mathbb N}$, then $ux$ denotes the one-sided sequence obtained by placing $u$ before $x$.

For $a\in\mathcal A$ and $n\in\mathbb N$, the notation $a^n$ denotes the finite word consisting of $n$ consecutive copies of $a$, while $a^\infty$ denotes the one-sided constant sequence $aaa\cdots$. Thus, for example, $0^2 1^3$ denotes the word $00111$.

\end{definition}

\begin{definition}\label{def:timesets}
Let $g\colon X\to X$ be continuous. For $\varepsilon>0$ and $x,y\in X$ define
\[
 N_{f_{1,\infty}}(x,y,\varepsilon)
 :=\bigl\{n\in\mathbb N_0: d\bigl(f_1^n(x),f_1^n(y)\bigr)<\varepsilon\bigr\},
\]
\[
 N_g(x,y,\varepsilon)
 :=\bigl\{n\in\mathbb N_0: d\bigl(g^n(x),g^n(y)\bigr)<\varepsilon\bigr\}.
\]
A set $A\subseteq\mathbb N_0$ is \emph{thick} if it contains intervals of arbitrary finite length, and \emph{syndetic} if there exists $L\ge 1$ such that every interval of length $L$ meets $A$. Its upper Banach density is
\[
 BD^*(A):=\limsup_{|I|\to\infty}\frac{|A\cap I|}{|I|},
\]
where $I$ ranges over finite intervals in $\mathbb N_0$. We say that $A$ has \emph{Banach density one} if $BD^*(\mathbb N_0\setminus A)=0$.
\end{definition}

\begin{definition}\label{def:prox-rp}\cite{LuChen2017}
The \emph{proximal relation} of $(X,f_{1,\infty})$ is
\[
 P_{f_{1,\infty}}
 :=\bigl\{(x,y)\in X\times X:\liminf_{n\to\infty} d\bigl(f_1^n(x),f_1^n(y)\bigr)=0\bigr\},
\]
and the proximal relation of $(X,g)$ is
\[
 P_g
 :=\bigl\{(x,y)\in X\times X:\liminf_{n\to\infty} d\bigl(g^n(x),g^n(y)\bigr)=0\bigr\}.
\]
\end{definition}

\begin{definition}\label{def:Qme}\cite{QiuZhao2020}
A pair $(x,y)\in X\times X$ belongs to $Q_{me}(g)$ if either $x=y$, or else for every $\tau>0$ there exists $c(\tau)>0$ such that for every $\varepsilon>0$ there are $u,v\in X$ and $n\in\mathbb N$ satisfying
\[
 d(u,v)<\varepsilon, \quad \text{and } \quad
 \frac1n\left|\Bigl\{0\le i\le n-1:
 d\bigl(g^i(u),x\bigr)<\tau,
 \ d\bigl(g^i(v),y\bigr)<\tau
\Bigr\}\right|>c(\tau).
\]
The relation $Q_{me}(g)$ is called the \emph{relation of sensitivity in the mean}.
\end{definition}

\begin{definition}\label{def:mean-equicontinuity}\cite{QiuZhao2020}
A continuous map $g\colon X\to X$ is \emph{mean equicontinuous} if for every $\varepsilon>0$ there exists $\delta>0$ such that $d(x,y)<\delta$ implies
\[
 \limsup_{n\to\infty}\frac1n\sum_{i=0}^{n-1}d\bigl(g^i(x),g^i(y)\bigr)<\varepsilon.
\]
\end{definition}

The following estimates for uniformly convergent sequences of maps are used in the paper. They are standard consequences of Wu--Zhu's finite-iterate estimates for uniformly convergent non-autonomous systems \cite[Lemma~2.1 and Corollary~2.2]{WuZhu2013}.

\begin{lemma}\label{lem:finite-window-convergence}
Let $(X,f_{1,\infty})$ be a non-autonomous dynamical system, and let $g\colon X\to X$ be a continuous map such that $f_n\to g$ uniformly. Then for every $L\in\mathbb N$ and every $\varepsilon>0$ there exists $N\in\mathbb N$ such that for all $i\ge N$ and all $0\le r\le L$, we have $\|f_i^r-g^r\|_\infty<\varepsilon$.

\end{lemma}

\begin{lemma}\label{lem:finite-window-tail-equicontinuity}
Let $(X,f_{1,\infty})$ be a non-autonomous dynamical system, and let $g\colon X\to X$ be a continuous map such that $f_n\to g$ uniformly. Then for every $L\in\mathbb N$ and every $\varepsilon>0$ there exist $N\in\mathbb N$ and $\delta>0$ such that for all $i\ge N$ and all $u,v\in X$ with $d(u,v)<\delta$,$d\bigl(f_i^r(u),f_i^r(v)\bigr)<\varepsilon
 \quad (0\le r\le L)$.

\end{lemma}

\section{Proximality, regional proximality, and mean relations}\label{sec:relations}

This section studies proximal-type relations for non-autonomous systems. We first treat uniformly convergent systems and then periodic systems, where the relations are determined by the period map. Examples show that suitable additional hypotheses are needed in the uniformly convergent case.

\begin{definition}\label{def:rp}
We define the \emph{regional proximal relation} of $(X,f_{1,\infty})$ by declaring $(x,y)\in RP_{f_{1,\infty}}$ if for every $\varepsilon>0$ there exist $u,v\in X$ and $n\ge 1$ such that
\[
 d(x,u)<\varepsilon,
 \qquad
 d(y,v)<\varepsilon,
 \qquad
 d\bigl(f_1^n(u),f_1^n(v)\bigr)<\varepsilon.
\]
\end{definition}

\begin{remark}\label{rem:seq-char-rp}
On compact metric spaces we use the usual sequential characterization of regional proximality: $(x,y)\in RP_{f_{1,\infty}}$ if and only if there exist $u_j\to x$, $v_j\to y$, and integers $n_j\ge 1$ such that
\[
 d\bigl(f_1^{n_j}(u_j),f_1^{n_j}(v_j)\bigr)\longrightarrow 0,
\]
and similarly for $RP_g$.
\end{remark}

\begin{definition}\label{def:bp}
Following the autonomous notion of Banach proximality, we define the \emph{Banach proximal relation} of $(X,f_{1,\infty})$ by
\[
 BP_{f_{1,\infty}}
 :=\bigl\{(x,y)\in X\times X: N_{f_{1,\infty}}(x,y,\varepsilon)\text{ has Banach density one for every }\varepsilon>0\bigr\},
\]
and the relation $BP_g$ of $(X,g)$ is defined analogously.
\end{definition}

\subsection{Uniformly convergent systems}

We study \(P_g\) and \(P_{f_{1,\infty}}\) by considering the sets of times at which the corresponding orbits are \(\varepsilon\)-close, in the case where \(f_n\to g\) uniformly. The next theorem combines the two basic thickness results for these close-time sets.

\begin{theorem}\label{thm:thick-close-times}
Let $(X,f_{1,\infty})$ be a non-autonomous dynamical system, and let $g\colon X\to X$ be a continuous map such that $f_n\to g$ uniformly. Then the following statements hold.
\begin{enumerate}
\item If $(X,f_{1,\infty})$ is uniformly rigid and $(x,y)\in P_g$, then for every $\varepsilon>0$ the set $N_{f_{1,\infty}}(x,y,\varepsilon)$ is thick. In particular, $P_g\subseteq P_{f_{1,\infty}}$.

\item If $(x,y)\in P_{f_{1,\infty}}$, then for every $\varepsilon>0$ the set $N_{f_{1,\infty}}(x,y,\varepsilon)$ is thick.
\end{enumerate}
\end{theorem}

\begin{proof}
(1) Fix $\varepsilon>0$ and $L\in\mathbb N$. By uniform continuity of the finite family $\{g^t:0\le t\le L\}$, there exists $\eta>0$ such that for all $u,v\in X$,
\[
d(u,v)<\eta \quad\text{implies}\quad d\bigl(g^t(u),g^t(v)\bigr)<\frac{\varepsilon}{5}
\qquad (0\le t\le L).
\]
Since $(x,y)\in P_g$, there exists $k\in\mathbb N_0$ such that $d\bigl(g^k(x),g^k(y)\bigr)<\eta$.

By Lemma~\ref{lem:finite-window-convergence}, applied with window length $k+L$, there exists $N\in\mathbb N$ such that for all $i\ge N$ and all $0\le r\le k+L$, we get $\|f_i^r-g^r\|_\infty<\frac{\varepsilon}{5}$.

By uniform continuity of the finite family $\{g^r:0\le r\le k+L\}$, choose $\delta>0$ such that for all $u,v\in X$,
\[
d(u,v)<\delta \quad\text{implies}\quad d\bigl(g^r(u),g^r(v)\bigr)<\frac{\varepsilon}{10}
\qquad (0\le r\le k+L).
\]
Choose a rigidity time $m\ge N$ such that $\|f_1^m-\id_X\|_\infty<\delta.$ Then, for every $0\le t\le L$,
\[
\begin{aligned}
d\bigl(f_1^{m+k+t}(x),f_1^{m+k+t}(y)\bigr)
&=d\bigl(f_{m+1}^{k+t}(f_1^m(x)),f_{m+1}^{k+t}(f_1^m(y))\bigr)\\
&\le d\bigl(f_{m+1}^{k+t}(f_1^m(x)),g^{k+t}(f_1^m(x))\bigr)\\
&\quad +d\bigl(g^{k+t}(f_1^m(x)),g^{k+t}(f_1^m(y))\bigr)\\
&\quad +d\bigl(g^{k+t}(f_1^m(y)),f_{m+1}^{k+t}(f_1^m(y))\bigr).
\end{aligned}
\]
The first and third terms are $<\varepsilon/5$. Moreover,
\[
\begin{aligned}
d\bigl(g^{k+t}(f_1^m(x)),g^{k+t}(f_1^m(y))\bigr)
&\le d\bigl(g^{k+t}(f_1^m(x)),g^{k+t}(x)\bigr)
   +d\bigl(g^{k+t}(x),g^{k+t}(y)\bigr)\\
&\quad +d\bigl(g^{k+t}(y),g^{k+t}(f_1^m(y))\bigr).
\end{aligned}
\]
The outer two terms are $<\varepsilon/10$ by the choice of $\delta$, while the middle term is
\[
d\bigl(g^t(g^k(x)),g^t(g^k(y))\bigr)<\frac{\varepsilon}{5}
\]
by the choice of $\eta$. Hence
\[
d\bigl(f_1^{m+k+t}(x),f_1^{m+k+t}(y)\bigr)<\varepsilon
\text{ for } 0\le t\le L \text{ and therefore }[m+k,m+k+L]\subseteq N_{f_{1,\infty}}(x,y,\varepsilon).
\]
Since $L$ was arbitrary, the set of $\varepsilon$-close times is thick.

\smallskip
\noindent
(2) Fix $\varepsilon>0$ and $L\in\mathbb N$. By Lemma~\ref{lem:finite-window-tail-equicontinuity}, there exist $N\in\mathbb N$ and $\delta>0$ such that whenever $i\ge N$ and $d(u,v)<\delta$, one has
\[
d\bigl(f_i^r(u),f_i^r(v)\bigr)<\varepsilon
\qquad (0\le r\le L).
\]
Since $(x,y)\in P_{f_{1,\infty}}$, there exists $n\ge N$ such that $d\bigl(f_1^n(x),f_1^n(y)\bigr)<\delta$. For $0\le t\le L$ we then have
\[
d\bigl(f_1^{n+t}(x),f_1^{n+t}(y)\bigr)
=d\bigl(f_{n+1}^t(f_1^n(x)),f_{n+1}^t(f_1^n(y))\bigr)<\varepsilon \text{ which implies } [n,n+L]\subseteq N_{f_{1,\infty}}(x,y,\varepsilon)
\]
and the conclusion follows.
\end{proof}

The next example shows that thick close-time sets need not have Banach density one.

\begin{example}\label{ex:prox-not-banach-prox}
Proximality does not imply Banach proximality, even for an autonomous system. Let $X=\{0,1\}^{\mathbb N}$ with the one-sided shift $\sigma$, and use the standard shift metric
\[
d(x,y)=
\begin{cases}
0, & x=y,\\
2^{-\min\{j\ge 1:x_j\ne y_j\}}, & x\ne y.
\end{cases}
\]
Consider the autonomous system $f_n=\sigma$ for every $n\ge 1$. Put
\[
x=000\cdots,
\qquad
 y=0^1 1^1 0^2 1^2 0^3 1^3\cdots.
\]
Then $(x,y)\in P_{f_{1,\infty}}$, because inside each zero block of length $k$ there are times $n$ for which the first $k$ coordinates of $\sigma^n(y)$ are all zero, and hence
\[
d\bigl(\sigma^n(x),\sigma^n(y)\bigr)\le 2^{-k}.
\]
However, $(x,y)\notin BP_{f_{1,\infty}}$. Indeed, for $\varepsilon=1/2$ the set $N_{f_{1,\infty}}(x,y,1/2)$ consists exactly of the times at which the first symbol of $\sigma^n(y)$ is $0$, and its complement contains the blocks of ones. Since those blocks have lengths $1,2,3,\dots$, the complement contains arbitrarily long intervals, so $N_{f_{1,\infty}}(x,y,1/2)$ does not have Banach density one.
\end{example}

The reverse transfer requires collective convergence, since the comparison length is not fixed.

\begin{proposition}\label{prop:Pf-to-RPg}
Let $(X,f_{1,\infty})$ be a non-autonomous dynamical system, and let $g\colon X\to X$ be a continuous map such that $f_n\to g$ uniformly. Suppose that $(X,f_{1,\infty})$ is uniformly rigid and that the family $\{f_i^r\}_{i,r\ge 1}$ converges collectively to $\{g^r\}_{r\ge 1}$. Then $P_{f_{1,\infty}}\subseteq RP_g$.
\end{proposition}

\begin{proof}
Let $(x,y)\in P_{f_{1,\infty}}$, and fix $\varepsilon>0$. By collective convergence, choose $N\in\mathbb N$ such that
\[
\|f_i^r-g^r\|_\infty<\frac{\varepsilon}{3}
\qquad (i\ge N,\ r\ge 1).
\]
Choose a rigidity time $m\ge N$ such that $\|f_1^m-\id_X\|_\infty<\frac{\varepsilon}{3}$.

Since $(x,y)\in P_{f_{1,\infty}}$, there exists $n>m$ such that $d\bigl(f_1^n(x),f_1^n(y)\bigr)<\frac{\varepsilon}{3}$. Put $r=n-m$, $x'=f_1^m(x)$, and $y'=f_1^m(y)$. Then $d(x,x')<\frac{\varepsilon}{3}$, $d(y,y')<\frac{\varepsilon}{3}$, and, since $m+1\ge N$,
\[
\begin{aligned}
d\bigl(g^r(x'),g^r(y')\bigr)
&\le d\bigl(g^r(x'),f_{m+1}^r(x')\bigr)
   +d\bigl(f_{m+1}^r(x'),f_{m+1}^r(y')\bigr)
   +d\bigl(f_{m+1}^r(y'),g^r(y')\bigr)\\
&= d\bigl(g^r(x'),f_{m+1}^r(x')\bigr)
   +d\bigl(f_1^n(x),f_1^n(y)\bigr)
   +d\bigl(f_{m+1}^r(y'),g^r(y')\bigr)\\
&<\varepsilon.
\end{aligned}
\]
Hence $(x,y)\in RP_g$.
\end{proof}

The next corollary gives a setting where the proximal relation is recovered from the limit map.

\begin{corollary}\label{cor:equicontinuous-limit-P-equality}
Let $(X,f_{1,\infty})$ be a non-autonomous dynamical system, and let $g\colon X\to X$ be a continuous map such that $f_n\to g$ uniformly. Suppose that $(X,f_{1,\infty})$ is uniformly rigid, that the family $\{f_i^r\}_{i,r\ge 1}$ converges collectively to $\{g^r\}_{r\ge 1}$, and that $(X,g)$ is equicontinuous. Then $P_{f_{1,\infty}}=P_g$.

\end{corollary}

\begin{proof}
By Theorem~\ref{thm:thick-close-times}(1), we have $P_g\subseteq P_{f_{1,\infty}}$. Conversely, Proposition~\ref{prop:Pf-to-RPg} gives $P_{f_{1,\infty}}\subseteq RP_g.$

It therefore suffices to prove that $RP_g\subseteq P_g$ when $g$ is equicontinuous. Let $(x,y)\in RP_g$ and fix $\varepsilon>0$. By equicontinuity, choose $\delta>0$ such that for all $a,b\in X$,
\[
d(a,b)<\delta \quad\text{implies}\quad d\bigl(g^n(a),g^n(b)\bigr)<\frac{\varepsilon}{3}
\qquad\text{for all }n\ge 0.
\]
Since $(x,y)\in RP_g$, there exist $u,v\in X$ and $k\ge 1$ such that
\[
d(x,u)<\delta,
\qquad
 d(y,v)<\delta,
\qquad
 d\bigl(g^k(u),g^k(v)\bigr)<\delta.
\]
Then, for every $n\ge 0$,
\[
\begin{aligned}
d\bigl(g^{n+k}(x),g^{n+k}(y)\bigr)
&\le d\bigl(g^{n+k}(x),g^{n+k}(u)\bigr)
   +d\bigl(g^{n+k}(u),g^{n+k}(v)\bigr)\\
&\quad +d\bigl(g^{n+k}(v),g^{n+k}(y)\bigr)\\
&<\frac{\varepsilon}{3}+\frac{\varepsilon}{3}+\frac{\varepsilon}{3}=\varepsilon.
\end{aligned}
\]
Since $\varepsilon>0$ was arbitrary, $(x,y)\in P_g$.
\end{proof}

The next example shows that even when uniform rigidity and collective convergence hold, one cannot expect $P_{f_{1,\infty}}=P_g$ in general.

\begin{example}\label{ex:nonsurjective-shift-prox}
This example shows that the equicontinuity assumption in Corollary~\ref{cor:equicontinuous-limit-P-equality} cannot simply be omitted. Let $X=\{0,1\}^{\mathbb N}$ with the standard shift metric as defined in example \ref{ex:prox-not-banach-prox} , let $g=\sigma$ be the one-sided shift, and define
\[
 h(x)=0x_1\,00x_1x_2\,000x_1x_2x_3\cdots.
\]
Set $f_1=h$, and $f_n=\sigma \quad (n\ge 2)$. Then $f_n\to g$ uniformly and collective convergence holds from time $2$ onward, since $f_i^r=\sigma^r=g^r$ for all $i\ge2$ and $r\ge1$. The system is uniformly rigid: if $c_k=k^2+1$, then $f_1^{c_k}(x)$ and $x$ agree on their first $k$ coordinates, so $\|f_1^{c_k}-\id_X\|_\infty\le 2^{-k}\to0$.

We claim that $P_{f_{1,\infty}}=X\times X$. Indeed, for each $k\ge 1$, let $b_k=k(k-1)+1$. Then for every $x\in X$ the point $f_1^{b_k}(x)=\sigma^{k(k-1)}(h(x))$ begins with the block $0^k$, because in the word $h(x)$ the block $0^k$ starts at position $k(k-1)+1$. Consequently, for every $x,y\in X$, $d\bigl(f_1^{b_k}(x),f_1^{b_k}(y)\bigr)\le 2^{-k}$, so every pair is proximal. 

On the other hand,$P_g\neq X\times X$; for example, the constant sequences $0^\infty$ and $1^\infty$ are not proximal under the shift. Thus, even in the uniformly rigid and collectively convergent setting, equality of proximal relations can fail without the extra equicontinuity hypothesis on the limit system.
\end{example}

The next theorem gives the main comparison for regional proximality in the uniformly convergent setting.

\begin{theorem}\label{thm:RP-comparison}
Let $(X,f_{1,\infty})$ be a non-autonomous dynamical system, and let $g\colon X\to X$ be a continuous map such that $f_n\to g$ uniformly. Suppose that $(X,f_{1,\infty})$ is uniformly rigid. Then $RP_g\subseteq RP_{f_{1,\infty}}$. If, in addition, the family $\{f_i^r\}_{i,r\ge 1}$ converges collectively to $\{g^r\}_{r\ge 1}$, then $RP_{f_{1,\infty}}=RP_g$.
\end{theorem}

\begin{proof}
First let $(x,y)\in RP_g$ and fix $\varepsilon>0$. By definition of $RP_g$, choose $u,v\in X$ and $k\ge 1$ such that
\[
d(x,u)<\frac{\varepsilon}{3},
\qquad
 d(y,v)<\frac{\varepsilon}{3},
\qquad
 d\bigl(g^k(u),g^k(v)\bigr)<\frac{\varepsilon}{5}.
\]
By uniform continuity of $g^k$, choose $\delta>0$ such that for all $a,b\in X$,
\[
d(a,b)<\delta \quad\text{implies}\quad d\bigl(g^k(a),g^k(b)\bigr)<\frac{\varepsilon}{5}.
\]
By Lemma~\ref{lem:finite-window-convergence}, applied with window length $k$, there exists $N\in\mathbb N$ such that for all $i\ge N$ and all $0\le r\le k$, we get $\|f_i^r-g^r\|_\infty<\frac{\varepsilon}{5}$. Choose a rigidity time $m\ge N$ such that $\|f_1^m-\id_X\|_\infty<\delta$. Then
\[
\begin{aligned}
d\bigl(f_1^{m+k}(u),f_1^{m+k}(v)\bigr)
&=d\bigl(f_{m+1}^k(f_1^m(u)),f_{m+1}^k(f_1^m(v))\bigr)\\
&\le \|f_{m+1}^k-g^k\|_\infty
   +d\bigl(g^k(f_1^m(u)),g^k(f_1^m(v))\bigr)
   +\|f_{m+1}^k-g^k\|_\infty.
\end{aligned}
\]
Moreover,
\[
\begin{aligned}
d\bigl(g^k(f_1^m(u)),g^k(f_1^m(v))\bigr)
&\le d\bigl(g^k(f_1^m(u)),g^k(u)\bigr)
   +d\bigl(g^k(u),g^k(v)\bigr)
   +d\bigl(g^k(v),g^k(f_1^m(v))\bigr)\\
&<\frac{\varepsilon}{5}+\frac{\varepsilon}{5}+\frac{\varepsilon}{5}=\frac{3\varepsilon}{5},
\end{aligned}
\]
because $d(f_1^m(u),u)<\delta$ and $d(f_1^m(v),v)<\delta$. Therefore $d\bigl(f_1^{m+k}(u),f_1^{m+k}(v)\bigr)<\varepsilon$.

Since $d(x,u)<\varepsilon$ and $d(y,v)<\varepsilon$, the points $u$ and $v$ together with the time $m+k$ witness that $(x,y)\in RP_{f_{1,\infty}}$.

Now suppose that collective convergence holds, and let $(x,y)\in RP_{f_{1,\infty}}$. By the sequential characterization, there exist $u_j\to x$, $v_j\to y$, and integers $n_j\ge 1$ such that $d\bigl(f_1^{n_j}(u_j),f_1^{n_j}(v_j)\bigr)\longrightarrow 0$.

Passing to a subsequence, we may suppose either that $(n_j)$ is bounded or that $n_j\to\infty$.

\smallskip
\noindent
\emph{Case 1: $(n_j)$ is bounded.}
Passing to a further subsequence, suppose that $n_j\equiv n$ is constant. Then continuity of $f_1^n$ gives $f_1^n(x)=f_1^n(y)$. Choose rigidity times $m_j>n$ such that $\|f_1^{m_j}-\id_X\|_\infty\longrightarrow 0$. Let $x_j=f_1^{m_j}(x)$ and  $y_j=f_1^{m_j}(y)$. Then $x_j\to x$ and $y_j\to y$. Moreover,
\[
x_j=f_{n+1}^{m_j-n}(f_1^n(x))=f_{n+1}^{m_j-n}(f_1^n(y))=y_j,
\]
so $d(g(x_j),g(y_j))=0$ for every $j$. Hence $(x,y)\in RP_g$.

\smallskip
\noindent
\emph{Case 2: $n_j\to\infty$.}
Let $(q_\ell)$ be a rigidity sequence with $q_\ell\to\infty$ and $\|f_1^{q_\ell}-\id_X\|_\infty\to0$. Passing to a subsequence of $(q_\ell)$, we may suppose $\|f_1^{q_\ell}-\id_X\|_\infty<\frac1\ell$
for $\ell\ge1$.

Since $n_j\to\infty$, choose strictly increasing indices $j_\ell$ such that $n_{j_\ell}>q_\ell$ for every $\ell$. Set
\[
m_\ell=q_\ell,
\qquad
 r_\ell=n_{j_\ell}-m_\ell,
\qquad
 x_\ell=f_1^{m_\ell}(u_{j_\ell}),
\qquad
 y_\ell=f_1^{m_\ell}(v_{j_\ell}).
\]
Then $m_\ell<n_{j_\ell}$, $m_\ell\to\infty$, $x_\ell\to x$, and $y_\ell\to y$. Also,
\[
\begin{aligned}
d\bigl(g^{r_\ell}(x_\ell),g^{r_\ell}(y_\ell)\bigr)
&\le d\bigl(g^{r_\ell}(x_\ell),f_{m_\ell+1}^{r_\ell}(x_\ell)\bigr)
  +d\bigl(f_{m_\ell+1}^{r_\ell}(x_\ell),f_{m_\ell+1}^{r_\ell}(y_\ell)\bigr)\\
&\quad +d\bigl(f_{m_\ell+1}^{r_\ell}(y_\ell),g^{r_\ell}(y_\ell)\bigr)\\
&= d\bigl(g^{r_\ell}(x_\ell),f_{m_\ell+1}^{r_\ell}(x_\ell)\bigr)
  +d\bigl(f_1^{n_{j_\ell}}(u_{j_\ell}),f_1^{n_{j_\ell}}(v_{j_\ell})\bigr)\\
&\quad +d\bigl(f_{m_\ell+1}^{r_\ell}(y_\ell),g^{r_\ell}(y_\ell)\bigr).
\end{aligned}
\]
The middle term tends to $0$ by construction, and the first and third terms tend to $0$ by collective convergence, because $m_\ell\to\infty$. Therefore
\[
d\bigl(g^{r_\ell}(x_\ell),g^{r_\ell}(y_\ell)\bigr)\longrightarrow 0,
\]
so $(x,y)\in RP_g$.
\end{proof}

The next example shows that uniform rigidity alone is not sufficient.

\begin{example}\label{ex:UR-no-CC-counterexample}
Uniform convergence and uniform rigidity do not suffice for the reverse transfer to $RP_g$. Let $X=[0,1]$ and define a sequence $(a_n)_{n\ge 0}$ by $a_0=1$ and
\[
N_0=0,
\qquad
N_k=2\sum_{j=1}^k j^2\qquad (k\ge 1),
\]
\[
a_n=
\begin{cases}
\exp\!\bigl((n-N_{k-1})/k\bigr), & N_{k-1}<n\le N_{k-1}+k^2,\\[1mm]
\exp\!\bigl((N_k-n)/k\bigr), & N_{k-1}+k^2<n\le N_k,
\end{cases}
\]
whenever $N_{k-1}<n\le N_k$. Define $f_n(x)=x^{a_n/a_{n-1}}$
for $n\ge 1$.

Then every $f_n$ is continuous and onto. Moreover $f_n\to g=\id_{[0,1]}$ uniformly: on the $k$-th block the exponents $a_n/a_{n-1}$ are $e^{1/k}$ or $e^{-1/k}$, hence tend to $1$, and the maps $x\mapsto x^c$ converge uniformly to $x\mapsto x$ on $[0,1]$ as $c\to1$.

The initial compositions telescope i.e., $f_1^n(x)=x^{a_n}$ for $n\ge 0$. Since $a_{N_k}=1$, we have $f_1^{N_k}=\id_{[0,1]}$ for $k\ge 1$, so the non-autonomous system is uniformly rigid. On the other hand, at the midpoint $M_k=N_{k-1}+k^2$ of the $k$-th block we have $a_{M_k}=e^k$, hence for every $x,y\in [0,1)$,
\[
f_1^{M_k}(x)=x^{e^k}\longrightarrow 0,
\qquad
f_1^{M_k}(y)=y^{e^k}\longrightarrow 0.
\]
Thus $[0,1)\times [0,1)\subseteq P_{f_{1,\infty}}$. But $g=\id_{[0,1]}$, so $RP_g=\Delta_{[0,1]}$. Therefore $P_{f_{1,\infty}}\nsubseteq RP_g$ and  $RP_{f_{1,\infty}}\neq RP_g$.

Finally, collective convergence fails. For $i_k=N_{k-1}+1$ and $j_k=k^2$, we get $f_{i_k}^{j_k}(x)=x^{a_{i_k+j_k-1}/a_{i_k-1}}=x^{e^k}$. Evaluating at \(x=1/2\), we get
\[
\|f_{i_k}^{j_k}-\id_{[0,1]}\|_\infty
\ge \frac12-2^{-e^k}
\ge \frac12-2^{-e}
>\frac14 .
\]
Hence collective convergence does not hold.
\end{example}

The next example shows that collective convergence alone is not sufficient.

\begin{example}\label{ex:CC-no-UR-counterexample}
Collective convergence does not suffice for transfer of proximal-type relations. Let $X=S^1=\mathbb R/\mathbb Z$ be endowed with the usual quotient metric, let $R_\alpha(x)=x+\alpha \pmod 1$ be an irrational rotation, fix $x_0\in S^1$, and define $f_1(x)=x_0$ and $f_n(x)=R_\alpha(x)$ for $n\ge 2$.

Then $f_n\to g=R_\alpha$ uniformly and $\{f_i^r\}_{i,r\ge 1}$ \emph{converge collectively} to $\{g^r\}_{r\ge 1}$, since for every $i\ge 2$ and $r\ge 1$ one has $f_i^r=g^r$. However, $f_1^n(x)=R_\alpha^{\,n-1}(x_0)$ for $n\ge 1$, which is independent of $x$. Hence every pair of points has identical forward orbit from time $1$ onward, and therefore $P_{f_{1,\infty}}=RP_{f_{1,\infty}}=X\times X$.

On the other hand, $g=R_\alpha$ is an isometry of $S^1$. If $(x,y)\in P_g$, then
\[
d(x,y)=\liminf_{n\to\infty}d\bigl(g^n(x),g^n(y)\bigr)=0,
\]
so $x=y$. Thus $P_g=\Delta_X$. The same argument gives $RP_g=\Delta_X$: if $(x,y)\in RP_g$, then for every $\varepsilon>0$ there exist $u,v\in X$ and $n\ge 1$ such that
\[
d(x,u)<\varepsilon,
\qquad d(y,v)<\varepsilon,
\qquad d\bigl(g^n(u),g^n(v)\bigr)<\varepsilon.
\]
Since $g$ is an isometry, $d(u,v)=d(g^n(u),g^n(v))<\varepsilon$, and hence $d(x,y)<3\varepsilon$. Letting $\varepsilon\to0$ gives $x=y$.

The system is not uniformly rigid. Indeed, for every $n\ge 1$ the map $f_1^n$ is constant, and the quotient metric on $S^1$ has diameter $1/2$, so $\|f_1^n-\id_X\|_\infty=\frac12$.

Thus the initial compositions cannot converge uniformly to the identity.
\end{example}

\subsection{Periodic systems and mean relations}

Throughout this subsection, let $(X,f_{1,\infty})$ be a periodic non-autonomous system with period $p$. For every $m\ge0$ and every $0\le r<p$, grouping the composition into $m$ complete period blocks followed by the initial block of length $r$ gives
\begin{equation}\label{eq:periodic-normal-form-rel}
f_1^{mp+r}=f_1^r\circ g^m.
\end{equation}

The next theorem shows that regional proximality in the periodic case is determined by the period map.

\begin{theorem}\label{prop:periodic-RP-equality}
Let $(X,f_{1,\infty})$ be periodic with period $p$, and let $g=f_1^p$. Then $RP_{f_{1,\infty}}=RP_g$.

\end{theorem}

\begin{proof}
The inclusion $RP_g\subseteq RP_{f_{1,\infty}}$ follows from \eqref{eq:periodic-normal-form-rel}, because $f_1^{mp}=g^m$ for every $m\ge0$.

For the reverse inclusion, let $(x,y)\in RP_{f_{1,\infty}}$. Choose sequences $u_j\to x$, $v_j\to y$, and integers $n_j\ge 1$ such that
\[
d\bigl(f_1^{n_j}(u_j),f_1^{n_j}(v_j)\bigr)\longrightarrow 0.
\]
There are only $p$ possible residues modulo $p$. Hence, by the pigeonhole principle, one residue occurs along an infinite subsequence. Passing to that subsequence, write $n_j=m_jp+r$, with $0\le r<p$.

If $(m_j)$ is bounded, pass to a further subsequence and suppose $m_j\equiv m$ is constant. By continuity, $f_1^{mp+r}(x)=f_1^{mp+r}(y)$.

If $r=0$, then $g^m(x)=g^m(y)$, so $(x,y)\in RP_g$. If $1\le r<p$, apply the continuous map $f_{r+1}^{p-r}$ to both sides. Since $f_{r+1}^{p-r}\circ f_1^r=g$, we get $g^{m+1}(x)=g^{m+1}(y)$, and again $(x,y)\in RP_g$.

Now suppose that $m_j\to\infty$. If $r=0$, then
\[
d\bigl(g^{m_j}(u_j),g^{m_j}(v_j)\bigr)=d\bigl(f_1^{n_j}(u_j),f_1^{n_j}(v_j)\bigr)\longrightarrow 0,
\]
and $(x,y)\in RP_g$. If $1\le r<p$, apply the continuous map $f_{r+1}^{p-r}$ to the two points $f_1^{n_j}(u_j)$ and $f_1^{n_j}(v_j)$. Since
\[
f_{r+1}^{p-r}\circ f_1^{n_j}
=f_{r+1}^{p-r}\circ f_1^r\circ g^{m_j}
=g^{m_j+1},
\]
we obtain
\[
d\bigl(g^{m_j+1}(u_j),g^{m_j+1}(v_j)\bigr)\longrightarrow 0.
\]
Hence $(x,y)\in RP_g$ in all cases.
\end{proof}

To handle Banach proximality, we first record a combinatorial lemma on Banach density under block expansion and block projection.

\begin{lemma}\label{lem:periodic-density-lemma}
Let $A\subseteq\mathbb N_0$ and let $p\ge 1$.
\begin{enumerate}
\item If $BD^*(A)=0$, then $pA+F=\{pa+r:a\in A,\ r\in F\}$ has zero upper Banach density for every finite set $F\subseteq\{0,1,\dots,p-1\}$.
\item If $A$ has Banach density one, then the set $\Gamma_p(A)=\bigl\{m\ge 1: A\cap[(m-1)p+1,mp]\neq\varnothing\bigr\}$ also has Banach density one.
\end{enumerate}
\end{lemma}

\begin{proof}
(1) Let $F\subseteq\{0,1,\dots,p-1\}$ be finite, and let $J=[a,b]\cap\mathbb N_0$ be a finite interval. Fix $r\in F$. The condition $pn+r\in J$ is equivalent to
\[
\frac{a-r}{p}\le n\le \frac{b-r}{p}.
\]
Thus the set of $n\in A$ with $pn+r\in J$ is contained in an interval of $\mathbb N_0$ with cardinality at most $|J|/p+2$. Therefore
\[
|(pA+F)\cap J|
\le |F|\,\sup\bigl\{|A\cap I|: I\subseteq\mathbb N_0\text{ is an interval and } |I|\le |J|/p+2\bigr\}.
\]
Dividing by $|J|$ and taking the limsup over intervals $J$ with $|J|\to\infty$ shows that $BD^*(pA+F)=0$ whenever $BD^*(A)=0$.

\smallskip
\noindent
(2) Let $B=\mathbb N\setminus \Gamma_p(A)$ be the set of bad blocks. For every $m\in B$, the whole block $[(m-1)p+1,mp]$ is disjoint from $A$. Let if possible $B$ have positive upper Banach density. Then there are intervals $K_j\subseteq\mathbb N$ with $|K_j|\to\infty$ and a number $\eta>0$ such that
\[
|B\cap K_j|\ge \eta |K_j|
\qquad (j\ge1).
\]
If $K_j=[s_j,t_j]\cap\mathbb N$, let $J_j=[(s_j-1)p+1,t_jp]\cap\mathbb N_0$.

Then $J_j$ is a finite interval with $|J_j|=p|K_j|$. Since every bad block indexed by $B\cap K_j$ is contained in $\mathbb N_0\setminus A$,
\[
|J_j\cap(\mathbb N_0\setminus A)|\ge p|B\cap K_j|\ge p\eta |K_j|=\eta |J_j|.
\]
This gives $BD^*(\mathbb N_0\setminus A)>0$, contradicting the assumption that $A$ has Banach density one. Hence $\Gamma_p(A)$ has Banach density one.
\end{proof}

The next theorem shows that Banach proximality in the periodic setting is likewise determined by the period map.

\begin{theorem}\label{prop:periodic-BP-equality}
Let $(X,f_{1,\infty})$ be periodic with period $p$, and let $g=f_1^p$. Then $BP_{f_{1,\infty}}=BP_g$.

\end{theorem}

\begin{proof}
We first prove $BP_g\subseteq BP_{f_{1,\infty}}$. Let $(x,y)\in BP_g$ and fix $\varepsilon>0$. By uniform continuity of the finite family $\{f_1^r:0\le r<p\}$, choose $\delta>0$ such that for all $u,v\in X$,
\[
d(u,v)<\delta \quad\text{implies}\quad d\bigl(f_1^r(u),f_1^r(v)\bigr)<\varepsilon
\qquad (0\le r<p).
\]
Hence, if $m\in N_g(x,y,\delta)$, then for every $0\le r<p$,
\[
d\bigl(f_1^{mp+r}(x),f_1^{mp+r}(y)\bigr)
=d\bigl(f_1^r(g^m(x)),f_1^r(g^m(y))\bigr)<\varepsilon.
\]
Thus
\[
p\cdot N_g(x,y,\delta)+\{0,1,\dots,p-1\}\subseteq N_{f_{1,\infty}}(x,y,\varepsilon).
\]
Since $N_g(x,y,\delta)$ has Banach density one, Lemma~\ref{lem:periodic-density-lemma}(1), applied to $\mathbb N_0\setminus N_g(x,y,\delta)$, implies that the complement of the left-hand side has zero upper Banach density. Hence $N_{f_{1,\infty}}(x,y,\varepsilon)$ has Banach density one.

For the reverse inclusion, let $(x,y)\in BP_{f_{1,\infty}}$ and fix $\varepsilon>0$. For each $1\le s<p$, let $h_s=f_{s+1}^{p-s}=f_p\circ\cdots\circ f_{s+1}$, and set $h_p=\id_X$. By uniform continuity of the finite family $\{h_s:1\le s\le p\}$, choose $\delta>0$ with $\delta<\varepsilon$ such that for all $u,v\in X$,
\[
d(u,v)<\delta \quad\text{implies}\quad d\bigl(h_s(u),h_s(v)\bigr)<\varepsilon
\qquad (1\le s\le p).
\]
Put $A=N_{f_{1,\infty}}(x,y,\delta)$. Since $(x,y)\in BP_{f_{1,\infty}}$, the set $A$ has Banach density one. By Lemma~\ref{lem:periodic-density-lemma}(2), the set
\[
\Gamma_p(A)=\bigl\{m\ge 1: A\cap[(m-1)p+1,mp]\neq\varnothing\bigr\}
\]
has Banach density one. We claim that $\Gamma_p(A)\subseteq N_g(x,y,\varepsilon)$. Indeed, let $m\in\Gamma_p(A)$, and choose $n\in A\cap[(m-1)p+1,mp]$. Write $n=(m-1)p+s$ with $1\le s\le p$. If $s=p$, then $n=mp$ and, by \eqref{eq:periodic-normal-form-rel},
\[
d\bigl(g^m(x),g^m(y)\bigr)=d\bigl(f_1^{mp}(x),f_1^{mp}(y)\bigr)<\delta<\varepsilon.
\]
If $1\le s<p$, then by \eqref{eq:periodic-normal-form-rel} we get $f_1^n=f_1^s\circ g^{m-1}$, so
\[
d\bigl(f_1^s(g^{m-1}(x)),f_1^s(g^{m-1}(y))\bigr)<\delta.
\]
Applying $h_s$ and using $h_s\circ f_1^s=g$, we obtain
\[
d\bigl(g^m(x),g^m(y)\bigr)
=d\bigl(h_s(f_1^s(g^{m-1}(x))),h_s(f_1^s(g^{m-1}(y)))\bigr)<\varepsilon.
\]
Therefore $\Gamma_p(A)\subseteq N_g(x,y,\varepsilon)$, and since $\Gamma_p(A)$ has Banach density one, so does $N_g(x,y,\varepsilon)$. Hence $(x,y)\in BP_g$.
\end{proof}

The final corollary links the periodic relation theory with the mean theory of the period map.

\begin{corollary}\label{cor:periodic-mean-eq-characterization}
Let $(X,f_{1,\infty})$ be periodic with period $p$, and let $g=f_1^p$. Then the following statements are equivalent:
\begin{enumerate}
\item $(X,g)$ is mean equicontinuous;
\item $RP_{f_{1,\infty}}=BP_{f_{1,\infty}}$;
\item $Q_{me}(g)=\Delta_X$.
\end{enumerate}
\end{corollary}

\begin{proof}
By Qiu--Zhao \cite{QiuZhao2020}, a continuous map $g$ is mean equicontinuous if and only if $RP_g=BP_g$, and this is also equivalent to $Q_{me}(g)=\Delta_X$. On the other hand, Theorems~\ref{prop:periodic-RP-equality} and \ref{prop:periodic-BP-equality} give $RP_{f_{1,\infty}}=RP_g$, and $BP_{f_{1,\infty}}=BP_g$ respectively. Therefore $RP_{f_{1,\infty}}=BP_{f_{1,\infty}}$
 if and only if $RP_g=BP_g$ if and only if $(X,g)$ is mean equicontinuous if and only if $Q_{me}(g)=\Delta_X$.

\end{proof}

\begin{remark}\label{rem:periodic-mean-sharpness}
The equivalence in Corollary~\ref{cor:periodic-mean-eq-characterization} is not a formal consequence of periodicity alone. In the autonomous subcase $p=1$, take $g$ to be the one-sided full shift. This system is not mean equicontinuous, and by the Qiu--Zhao characterization one has $RP_g\ne BP_g$ and $Q_{me}(g)\ne\Delta_X$. Thus the corollary gives a genuine criterion rather than an automatic identity. Outside the periodic regime there is no comparable reduction to the limit map: Examples~\ref{ex:UR-no-CC-counterexample}, \ref{ex:CC-no-UR-counterexample}, and \ref{ex:nonsurjective-shift-prox} exhibit uniformly convergent systems whose non-autonomous proximal or regional proximal relations are much larger than the corresponding relations of $g$.
\end{remark}

\section{Maximal equicontinuous factors and their kernels}\label{sec:mef}

We now study maximal equicontinuous factors for non-autonomous systems. In the autonomous theory, maximal equicontinuous factors and their kernel relations are classical objects; see Veech \cite{Veech1968}, Ellis and Keynes \cite{EllisKeynes1971}, and Auslander et al. \cite{AuslanderEllisEllis1995}. For a non-autonomous system, the corresponding factor theory must take into account the entire sequence of generators. Accordingly, this section considers factor maps that intertwine each generator individually and develops the associated maximal equicontinuous factor in that category.

The maximal equicontinuous factor is first constructed by the standard quotient/product argument. Under collective convergence and uniform rigidity, its kernel is then described as the smallest closed equivalence relation containing the autonomous structure relation of the limit map and preserved by all generators.

\subsection{The maximal equicontinuous factor}

We begin by fixing the factor notion that will be used throughout the section.

\begin{definition}\label{def:strict-eq-factor}
Let $(X,f_{1,\infty})$ be a non-autonomous dynamical system. An \emph{equicontinuous factor} of $(X,f_{1,\infty})$ is a factor map $\pi\colon (X,f_{1,\infty})\longrightarrow (Y,h_{1,\infty})$,

where $Y$ is a compact metric space, $h_{1,\infty}=(h_n)_{n\ge 1}$ is a sequence of continuous self-maps of $Y$, and the system $(Y,h_{1,\infty})$ is equicontinuous. Its kernel relation is
\[
 \ker\pi
 :=\bigl\{(x,y)\in X\times X:\pi(x)=\pi(y)\bigr\}.
\]

A \emph{maximal equicontinuous factor} of $(X,f_{1,\infty})$ is an equicontinuous factor
\[
 \pi_{eq}\colon (X,f_{1,\infty})\longrightarrow
 (X_{eq},h_{1,\infty}^{eq}),
\]
where $(X_{eq},h_{1,\infty}^{eq})$ is an equicontinuous non-autonomous system, such that for every equicontinuous factor $\pi\colon (X,f_{1,\infty})\longrightarrow (Y,h_{1,\infty})$, there exists a unique factor map $\rho\colon (X_{eq},h_{1,\infty}^{eq})\longrightarrow (Y,h_{1,\infty})$ satisfying $\pi=\rho\circ \pi_{eq}$.

\end{definition}

The next proposition records the existence of the maximal equicontinuous factor in this setting.

\begin{proposition}\label{prop:strict-mef-existence}
Let $(X,f_{1,\infty})$ be a non-autonomous dynamical system. Then $(X,f_{1,\infty})$ admits a maximal equicontinuous factor.
\end{proposition}

\begin{proof}
Let $\mathcal R_{eq}$ be the set of all closed equivalence relations $R\subseteq X\times X$ for which the quotient map
\[
 \pi_R\colon X\to X/R
\]
realizes $X/R$ as an equicontinuous factor of $(X,f_{1,\infty})$. This collection is nonempty, since $X\times X$ gives the one-point equicontinuous factor. Since $\mathcal R_{eq}\subseteq \mathcal P(X\times X)$, the collection is a set. For each $R\in\mathcal R_{eq}$, choose one corresponding factor system
\[
 \pi_R\colon (X,f_{1,\infty})\to (X/R,h_{1,\infty}^R).
\]
Let $S_{eq}(f_{1,\infty})=\bigcap_{R\in\mathcal R_{eq}} R$. An intersection of closed equivalence relations is again a closed equivalence relation. Hence $S_{eq}(f_{1,\infty})$ is a closed equivalence relation on $X$.

Consider the diagonal map
\[
 \Pi\colon X\to \prod_{R\in\mathcal R_{eq}} X/R,
 \qquad
 \Pi(x)=\bigl(\pi_R(x)\bigr)_{R\in\mathcal R_{eq}},
\]
and let $Y=\Pi(X)$. For each $n\ge 1$, define
\[
 H_n\bigl((y_R)_{R\in\mathcal R_{eq}}\bigr)=\bigl(h_n^R(y_R)\bigr)_{R\in\mathcal R_{eq}}.
\]
Then $H_n(Y)\subseteq Y$ because each coordinate satisfies $\pi_R\circ f_n=h_n^R\circ \pi_R$.

Thus $\Pi\colon (X,f_{1,\infty})\to (Y,H_{1,\infty})$ is a factor map onto $Y$. The product of equicontinuous systems is equicontinuous with respect to the product uniformity, and therefore its invariant subsystem $(Y,H_{1,\infty})$ is equicontinuous.

Moreover,
\[
 \ker \Pi=\bigcap_{R\in\mathcal R_{eq}} \ker\pi_R
 =\bigcap_{R\in\mathcal R_{eq}} R
 =S_{eq}(f_{1,\infty}).
\]
Therefore $Y$ is canonically homeomorphic to the quotient $X/S_{eq}(f_{1,\infty})$. Since $X$ is compact metric and $S_{eq}(f_{1,\infty})$ is closed, the quotient is compact metrizable. Hence the induced quotient system is an equicontinuous factor of the original system.

Finally, let $\rho\colon (X,f_{1,\infty})\to (Z,k_{1,\infty})$ be any equicontinuous factor. Then $\ker \rho\in\mathcal R_{eq}$, so $S_{eq}(f_{1,\infty})\subseteq \ker \rho$. Hence $\rho$ factors uniquely through the quotient by $S_{eq}(f_{1,\infty})$. This proves the universal property.
\end{proof}

\begin{remark}\label{rem:Seq-kernel-definition}
By Proposition~\ref{prop:strict-mef-existence}, the kernel of the maximal equicontinuous factor is precisely the relation
\[
 S_{eq}(f_{1,\infty})
 =\bigcap\bigl\{\ker\pi: \pi \text{ is an equicontinuous factor map of }(X,f_{1,\infty})\bigr\}.
\]
This notation will be used from now on.
\end{remark}

We next show that every factor carries a continuous map induced by the limit map.

\begin{theorem}\label{thm:shadow-under-CC}
Let $(X,f_{1,\infty})$ be a non-autonomous dynamical system, and let $g\colon X\to X$ be a continuous map such that $f_n\to g$ uniformly. Suppose that the family $\{f_i^r\}_{i,r\ge 1}$ converges collectively to $\{g^r\}_{r\ge 1}$. Let $\pi\colon (X,f_{1,\infty})\to (Y,h_{1,\infty})$ be a factor map. Then the following statements hold.
\begin{enumerate}
\item $\ker\pi$ is preserved by $g$;
\item there exists a unique continuous map $\tau\colon Y\to Y$ such that $\pi\circ g=\tau\circ \pi$;

\item $h_n\to \tau$ uniformly;
\item $\{h_i^r\}_{i,r\ge 1}$ converges collectively to $\{\tau^r\}_{r\ge 1}$; that is,
\[
 \|h_i^r-\tau^r\|_\infty\longrightarrow 0
 \quad\text{uniformly in } r\ge 1 \text{ as } i\to\infty.
\]
\end{enumerate}
\end{theorem}

\begin{proof}
(1) Let $(x,y)\in\ker\pi$. Then $\pi(x)=\pi(y)$ and for every $n\ge 1$,
\[
 \pi(f_n(x))=h_n(\pi(x))=h_n(\pi(y))=\pi(f_n(y)).
\]
Since $f_n\to g$ uniformly and $\pi$ is continuous,
\[
 \pi(g(x))=\lim_{n\to\infty}\pi(f_n(x))=
 \lim_{n\to\infty}\pi(f_n(y))=\pi(g(y)).
\]
Thus $\ker\pi$ is preserved by $g$.

(2) Define $\tau\colon Y\to Y$ by $\tau(\pi(x)):=\pi(g(x))$  for $x\in X$.The preceding paragraph shows that this is well-defined. The identity $\tau\circ\pi=\pi\circ g$ implies continuity of $\tau$, because $\pi$ is a quotient map. Uniqueness follows because $\pi$ is onto.

(3) To prove uniform convergence, fix $\varepsilon>0$. Let $d_X$ and $d_Y$ denote the metrics on $X$ and $Y$, respectively. By uniform continuity of $\pi$, there exists $\delta>0$ such that for all $u,v\in X$,
\[
  d_X(u,v)<\delta \quad\text{implies}\quad d_Y\bigl(\pi(u),\pi(v)\bigr)<\varepsilon.
\]
Choose $N$ such that $\|f_n-g\|_\infty<\delta$ for all $n\ge N$. For any $y\in Y$, choose $x\in X$ with $\pi(x)=y$. If $n\ge N$, then $d_X(f_n(x),g(x))<\delta$, and hence
\[
 d_Y\bigl(h_n(y),\tau(y)\bigr)
 =d_Y\bigl(\pi(f_n(x)),\pi(g(x))\bigr)<\varepsilon.
\]
Hence $\|h_n-\tau\|_\infty<\varepsilon$ for all $n\ge N$, so $h_n\to \tau$ uniformly.

(4) By induction on $r$ we have $\pi\circ g^r=\tau^r\circ \pi$, for $r\ge 1$. Fix $\varepsilon>0$. By uniform continuity of $\pi$, choose $\delta>0$ such that for all $u,v\in X$,
\[
 d_X(u,v)<\delta \quad\text{implies}\quad d_Y\bigl(\pi(u),\pi(v)\bigr)<\varepsilon.
\]
By collective convergence above, choose $N$ such that for all $i\ge N$ and all $r\ge 1$, $\|f_i^r-g^r\|_\infty<\delta$.

Now let $i\ge N$, $r\ge 1$, and $y\in Y$. Choose $x\in X$ with $\pi(x)=y$. Since $d_X(f_i^r(x),g^r(x))<\delta$, we get
\[
 d_Y\bigl(h_i^r(y),\tau^r(y)\bigr)
 =d_Y\bigl(\pi(f_i^r(x)),\pi(g^r(x))\bigr)<\varepsilon.
\]
Taking the supremum over $y\in Y$ yields $\|h_i^r-\tau^r\|_\infty<\varepsilon$, uniformly in $r\ge 1$.
\end{proof}

The next proposition records a simple transfer principle for equicontinuity.

\begin{proposition}\label{prop:eq-shadow-to-nads}
Let $(Y,h_{1,\infty})$ be a non-autonomous dynamical system, and let $\tau\colon Y\to Y$ be continuous. Suppose that $\{h_i^r\}_{i,r\ge 1}$ converges collectively to $\{\tau^r\}_{r\ge 1}$. If $(Y,\tau)$ is equicontinuous, then $(Y,h_{1,\infty})$ is equicontinuous.
\end{proposition}

\begin{proof}
Fix $\varepsilon>0$. By equicontinuity of $(Y,\tau)$, choose $\eta>0$ such that for all $u,v\in Y$,
\[
 d_Y(u,v)<\eta \quad\text{implies}\quad d_Y\bigl(\tau^r(u),\tau^r(v)\bigr)<\frac{\varepsilon}{3}
 \qquad \text{for all } r\ge 0.
\]
By collective convergence, choose $N\in\mathbb N$ such that
\[
 \|h_{N+1}^r-\tau^r\|_\infty<\frac{\varepsilon}{3}
 \qquad\text{for all } r\ge 1.
\]
Because the finite family $\{h_1^j:0\le j\le N\}$ is equicontinuous, choose $\delta>0$ such that for all $y_1,y_2\in Y$,
\[
 d_Y(y_1,y_2)<\delta \quad\text{implies}\quad
 d_Y\bigl(h_1^j(y_1),h_1^j(y_2)\bigr)<\min\{\varepsilon,\eta\}
 \qquad (0\le j\le N).
\]
Now let $d_Y(y_1,y_2)<\delta$. For $n\le N$ the desired estimate already holds. If $n=N+r$ with $r\ge 1$, then $h_1^{N+r}=h_{N+1}^r\circ h_1^N$, so
\[
 \begin{aligned}
 d_Y\bigl(h_1^{N+r}(y_1),h_1^{N+r}(y_2)\bigr)
 &\le \|h_{N+1}^r-\tau^r\|_\infty
     +d_Y\bigl(\tau^r(h_1^N(y_1)),\tau^r(h_1^N(y_2))\bigr)\\
 &\quad +\|h_{N+1}^r-\tau^r\|_\infty \\
 &< \frac{\varepsilon}{3}+\frac{\varepsilon}{3}+\frac{\varepsilon}{3}=\varepsilon.
 \end{aligned}
\]
Hence $\{h_1^n:n\ge 0\}$ is equicontinuous.
\end{proof}

\begin{lemma}\label{lem:ur-factors}
Let $\pi\colon (X,f_{1,\infty})\longrightarrow (Y,h_{1,\infty})$ be a factor map between non-autonomous systems. If $(X,f_{1,\infty})$ is uniformly rigid, then $(Y,h_{1,\infty})$ is uniformly rigid.
\end{lemma}

\begin{proof}
Let $(n_k)$ be a rigidity sequence for $(X,f_{1,\infty})$. Since $\pi$ is uniformly continuous, for every $\varepsilon>0$ there is $\delta>0$ such that for all $x_1,x_2\in X$,
\[
 d_X(x_1,x_2)<\delta \quad\text{implies}\quad d_Y(\pi(x_1),\pi(x_2))<\varepsilon.
\]
Choose $k$ with $\|f_1^{n_k}-\id_X\|_\infty<\delta$. Let $y\in Y$. By surjectivity of $\pi$, there exists $x\in X$ such that $\pi(x)=y$. Since $\pi\circ f_n=h_n\circ\pi$ for all $n$, induction gives $h_1^{n_k}\circ\pi=\pi\circ f_1^{n_k}$. Hence
\[
 h_1^{n_k}(y)=h_1^{n_k}(\pi(x))=\pi(f_1^{n_k}(x)),
\]
and therefore
\[
 d_Y(h_1^{n_k}(y),y)=d_Y(\pi(f_1^{n_k}(x)),\pi(x))<\varepsilon.
\]
Taking the supremum over $y\in Y$ gives $\|h_1^{n_k}-\id_Y\|_\infty<\varepsilon$.
\end{proof}

The next theorem relates equicontinuity of a non-autonomous factor to equicontinuity of the factor induced by the limit map.

\begin{theorem}\label{thm:eq-factor-to-eq-shadow}
Let $(X,f_{1,\infty})$ be a non-autonomous dynamical system, and let $g\colon X\to X$ be a continuous map such that $f_n\to g$ uniformly. Suppose that the family $\{f_i^r\}_{i,r\ge 1}$ converges collectively to $\{g^r\}_{r\ge 1}$ and that $(X,f_{1,\infty})$ is uniformly rigid. Let $\pi\colon (X,f_{1,\infty})\to (Y,h_{1,\infty})$ be an equicontinuous factor, and let $\tau\colon Y\to Y$ be the continuous map induced by $g$ as in Theorem~\ref{thm:shadow-under-CC}. Then $(Y,\tau)$ is equicontinuous.
\end{theorem}

\begin{proof}
By Lemma~\ref{lem:ur-factors}, the factor system $(Y,h_{1,\infty})$ is uniformly rigid. Let $(n_k)$ be a rigidity sequence for $(Y,h_{1,\infty})$, so that $\|h_1^{n_k}-\id_Y\|_\infty\longrightarrow 0$.Fix $m\ge 0$. If $m=0$, then $\tau^0=\id_Y$ and we get $\|h_1^{n_k}-\tau^0\|_\infty\longrightarrow0.$

Now let $m\ge1$. By Theorem~\ref{thm:shadow-under-CC}, $\|h_{n_k+1}^m-\tau^m\|_\infty\longrightarrow 0$, because $n_k+1\to\infty$. Since $h_1^{n_k+m}=h_{n_k+1}^m\circ h_1^{n_k}$, we claim that $\|h_1^{n_k+m}-\tau^m\|_\infty\longrightarrow 0.$ Indeed, fix $\varepsilon>0$. By uniform continuity of $\tau^m$, choose $\delta>0$ such that for all $a,b\in Y$,
\[
 d_Y(a,b)<\delta \quad\text{implies}\quad d_Y\bigl(\tau^m(a),\tau^m(b)\bigr)<\frac{\varepsilon}{2}.
\]
Choose $k$ so large that
\[
 \|h_1^{n_k}-\id_Y\|_\infty<\delta
 \qquad\text{and}\qquad
 \|h_{n_k+1}^m-\tau^m\|_\infty<\frac{\varepsilon}{2}.
\]
Then for every $y\in Y$,
\[
 \begin{aligned}
 d_Y\bigl(h_1^{n_k+m}(y),\tau^m(y)\bigr)
 &=d_Y\bigl(h_{n_k+1}^m(h_1^{n_k}(y)),\tau^m(y)\bigr)\\
 &\le d_Y\bigl(h_{n_k+1}^m(h_1^{n_k}(y)),\tau^m(h_1^{n_k}(y))\bigr)
    +d_Y\bigl(\tau^m(h_1^{n_k}(y)),\tau^m(y)\bigr)\\
 &<\frac{\varepsilon}{2}+\frac{\varepsilon}{2}=\varepsilon.
 \end{aligned}
\]
This proves the claim. Consequently every iterate $\tau^m$ belongs to the uniform closure of the family of initial compositions $\{h_1^n:n\ge 0\}$.

It remains to show that the uniform closure of an equicontinuous family on a compact metric space is again equicontinuous. Indeed, fix $\varepsilon>0$. Since $(Y,h_{1,\infty})$ is equicontinuous, there exists $\delta>0$ such that for all $y_1,y_2\in Y$,
\[
 d_Y(y_1,y_2)<\delta \quad\text{implies}\quad d_Y\bigl(h_1^n(y_1),h_1^n(y_2)\bigr)<\frac{\varepsilon}{3}
 \qquad (n\ge 0).
\]
If $T$ lies in the uniform closure of $\{h_1^n:n\ge 0\}$, choose $n$ with $\|T-h_1^n\|_\infty<\varepsilon/3$. Then
\[
d_Y\bigl(T(y_1),T(y_2)\bigr)
 \le d_Y\bigl(T(y_1),h_1^n(y_1)\bigr)
    +d_Y\bigl(h_1^n(y_1),h_1^n(y_2)\bigr)
    +d_Y\bigl(h_1^n(y_2),T(y_2)\bigr)
 <\varepsilon, 
\]
whenever $d_Y(y_1,y_2)<\delta$. Thus the closure is equicontinuous. Applying this to the family $\{h_1^n:n\ge 0\}$ shows that $\{\tau^m:m\ge 0\}$ is equicontinuous, i.e. $(Y,\tau)$ is equicontinuous.
\end{proof}

\begin{example}\label{ex:UR-needed-for-shadow}
The uniform rigidity hypothesis in Theorem~\ref{thm:eq-factor-to-eq-shadow} is essential. Let $X=S^1=\mathbb R/\mathbb Z$, define $f_1(x)=0$ and $f_n(x)=2x \pmod 1$ for $n\ge 2$, and let $g(x)=2x\pmod 1$. Since $f_n=g$ for all $n\ge2$, both uniform convergence and collective convergence hold. Moreover, for every $n\ge1$, $f_1^n(x)=0$, so the family of initial compositions consists of $\id_X$ together with constant maps, and is therefore equicontinuous. By contrast, the doubling map is not equicontinuous on $S^1$: for any $\delta>0$, one may choose $x,y\in S^1$ with $0<d(x,y)<\delta$ and then an iterate $m$ such that $d(g^m(x),g^m(y))>1/4$. Thus collective convergence alone does not force the autonomous map induced by the limit on an equicontinuous non-autonomous factor to be equicontinuous.
\end{example}

The next corollary says that, in the uniformly rigid regime, every equicontinuous factor contains the autonomous structure relation of the limit map.

\begin{corollary}\label{cor:strict-factor-kernel-contains-Sg}
Let $(X,f_{1,\infty})$ be a non-autonomous dynamical system, and let $g\colon X\to X$ be a continuous map such that $f_n\to g$ uniformly. Suppose that the family $\{f_i^r\}_{i,r\ge 1}$ converges collectively to $\{g^r\}_{r\ge 1}$ and that $(X,f_{1,\infty})$ is uniformly rigid. Let $\pi_g\colon (X,g)\to (X_g^{eq},T)$ be the autonomous maximal equicontinuous factor of $g$, and let $S_g=\ker\pi_g$ be the autonomous equicontinuous structure relation of $(X,g)$. If $\pi\colon (X,f_{1,\infty})\to (Y,h_{1,\infty})$ is an equicontinuous factor, then $S_g\subseteq \ker\pi$.
\end{corollary}

\begin{proof}
By Theorem~\ref{thm:shadow-under-CC}, the factor admits a continuous map $\tau$ induced by $g$, and by Theorem~\ref{thm:eq-factor-to-eq-shadow}, the autonomous system $(Y,\tau)$ is equicontinuous. Thus $\pi\colon (X,g)\to (Y,\tau)$ is an equicontinuous factor of $(X,g)$. By the universal property of the autonomous maximal equicontinuous factor, its kernel contains $S_g$.
\end{proof}

\subsection{Kernel classification and preservation criteria}

We now identify the closed equivalence relations that arise as kernels of equicontinuous factors in the uniformly rigid and collectively convergent setting.

\begin{theorem}\label{thm:kernel-classification}
Let $(X,f_{1,\infty})$ be a non-autonomous dynamical system, and let $g\colon X\to X$ be a continuous map such that $f_n\to g$ uniformly. Suppose that the family $\{f_i^r\}_{i,r\ge 1}$ converges collectively to $\{g^r\}_{r\ge 1}$ and that $(X,f_{1,\infty})$ is uniformly rigid. Let $\pi_g\colon (X,g)\to (X_g^{eq},T)$ be the autonomous maximal equicontinuous factor of $g$, and let $S_g=\ker\pi_g$.
\begin{enumerate}
\item A closed equivalence relation $R\subseteq X\times X$ is the kernel of an equicontinuous factor of $(X,f_{1,\infty})$ if and only if $S_g\subseteq R$ and $(f_n\times f_n)(R)\subseteq R$ for every $n\ge 1$.

\item $S_{eq}(f_{1,\infty})$ is the smallest closed equivalence relation on $X$ that contains $S_g$ and is preserved by every generator $f_n$. In particular, if $S_g$ itself is preserved by every $f_n$, then $S_{eq}(f_{1,\infty})=S_g$. Equivalently, the quotient by the autonomous maximal equicontinuous factor relation of $g$ is already the maximal equicontinuous factor of the non-autonomous system.
\item Let $\mathcal W^*$ be the set of finite words in the alphabet $\mathbb N$, together with the empty word $\varnothing$. For a word $w=(i_1,\dots,i_k)$, let $f_w=f_{i_k}\circ\cdots\circ f_{i_1}$, and let $f_{\varnothing}=\id_X$. Define
\[
 \Sigma_g=\bigcup_{w\in\mathcal W^*}(f_w\times f_w)(S_g).
\]
Then $S_{eq}(f_{1,\infty})$ is the smallest closed equivalence relation on $X$ containing $\Sigma_g$.
\end{enumerate}
\end{theorem}

\begin{proof}
(1) First suppose that $R=\ker\pi$ for some equicontinuous factor
\[
 \pi\colon (X,f_{1,\infty})\to (Y,h_{1,\infty}).
\]
Since $\pi$ is a factor, $R$ is preserved by every generator, i.e., $(f_n\times f_n)(R)\subseteq R$ for $n\ge 1$. By Corollary~\ref{cor:strict-factor-kernel-contains-Sg}, we also have $S_g\subseteq R$.

Conversely, suppose that $R$ is a closed equivalence relation containing $S_g$ and preserved by every $f_n$. Let $\pi_R\colon X\to X/R$ be the quotient map. For each $n\ge1$, define $\bar f_n(\pi_R(x))=\pi_R(f_n(x))$. This is well-defined because $R$ is preserved by $f_n$. Moreover, $\bar f_n\circ \pi_R=\pi_R\circ f_n$, so $\bar f_n$ is continuous by the quotient property of $\pi_R$. Hence $\pi_R\colon (X,f_{1,\infty})\to (X/R,\bar f_{1,\infty})$ is a factor map.

Since $S_g\subseteq R$, the map $\rho(\pi_g(x)):=\pi_R(x)$ is well-defined. It is a continuous surjection $\rho\colon X_g^{eq}\to X/R$ satisfying $\pi_R=\rho\circ \pi_g$. Thus $\pi_R$ factors through the autonomous maximal equicontinuous factor map $\pi_g$.

Let $\tau_R$ be the map induced by $g$ on $X/R$ as in Theorem~\ref{thm:shadow-under-CC}. For every $x\in X$,
\[
 \tau_R\bigl(\rho(\pi_g(x))\bigr)
 =\tau_R(\pi_R(x))
 =\pi_R(g(x))
 =\rho(\pi_g(g(x)))
 =\rho\bigl(T(\pi_g(x))\bigr).
\]
Thus $\tau_R\circ \rho=\rho\circ T$, so \(\rho\colon (X_g^{eq},T)\to (X/R,\tau_R)\) is an autonomous factor map. Since \((X_g^{eq},T)\) is equicontinuous and \(\rho\) is a continuous surjection between compact metric spaces, equicontinuity passes to the factor. Hence \((X/R,\tau_R)\) is equicontinuous. By Theorem~\ref{thm:shadow-under-CC}, collective convergence passes from $(X,f_{1,\infty})$ to the quotient factor $(X/R,\bar f_{1,\infty})$, and then Proposition~\ref{prop:eq-shadow-to-nads} implies that $(X/R,\bar f_{1,\infty})$ is equicontinuous. Thus $R$ is the kernel of an equicontinuous factor.

\smallskip
\noindent
(2) By part~(1), the kernels of equicontinuous factors are exactly the closed equivalence relations containing $S_g$ and preserved by every generator. Since $S_{eq}(f_{1,\infty})$ is the intersection of all such kernels, it is the smallest relation with these two properties. The final assertion follows by taking $R=S_g$.

\smallskip
\noindent
(3) By part~(2), $S_{eq}(f_{1,\infty})$ contains $(f_w\times f_w)(S_g)$ for every finite word $w$, and hence contains $\Sigma_g$.

Conversely, let $R_\Sigma$ be the smallest closed equivalence relation containing $\Sigma_g$. We claim that $R_\Sigma$ is preserved by every generator $f_n$. First, the generating set $\Sigma_g$ is forward invariant: if $w\in\mathcal W^*$ and $n\ge1$, then
\[
 (f_n\times f_n)\bigl((f_w\times f_w)(S_g)\bigr)=(f_{w\ast n}\times f_{w\ast n})(S_g)\subseteq \Sigma_g,
\]
where $w\ast n$ denotes the word obtained by appending $n$ to $w$. The algebraic equivalence relation generated by $\Sigma_g$ consists of finite chains whose links lie in $\Sigma_g$ or in the inverse relation $\Sigma_g^{-1}$. Since $f_n\times f_n$ sends each such link to another link of the same type, it preserves this algebraic equivalence relation. Finally, $f_n\times f_n$ is continuous, so it preserves the closure $R_\Sigma$. Thus $R_\Sigma$ is a closed equivalence relation containing $S_g$ and preserved by every generator $f_n$. By part~(2),
\[
 S_{eq}(f_{1,\infty})\subseteq R_\Sigma.
\]
The reverse inclusion was already proved, so equality holds.
\end{proof}

The next corollary shows how the maximal equicontinuous factor description simplifies once the autonomous equicontinuous structure relation is already known.

\begin{corollary}\label{cor:RP-kernel-version}
Let $(X,f_{1,\infty})$ be a non-autonomous dynamical system, and let $g\colon X\to X$ be a continuous map such that $f_n\to g$ uniformly. Suppose that the family $\{f_i^r\}_{i,r\ge 1}$ converges collectively to $\{g^r\}_{r\ge 1}$ and that $(X,f_{1,\infty})$ is uniformly rigid. Let $\pi_g\colon (X,g)\to (X_g^{eq},T)$ be the autonomous maximal equicontinuous factor of $g$, and put $S_g=\ker\pi_g$. Suppose moreover that $S_g=RP_g$. If each $f_n$ preserves $RP_g$, then $S_{eq}(f_{1,\infty})=RP_g=RP_{f_{1,\infty}}$. In particular, the quotient by $RP_{f_{1,\infty}}$ is the maximal equicontinuous factor of $(X,f_{1,\infty})$.
\end{corollary}

\begin{proof}
Since $S_g=RP_g$ and each $f_n$ preserves $RP_g$, Theorem~\ref{thm:kernel-classification}(2) gives $S_{eq}(f_{1,\infty})=S_g=RP_g.$ By Theorem~\ref{thm:RP-comparison}, proved in Section~\ref{sec:relations}, $RP_{f_{1,\infty}}=RP_g.$ Combining the two equalities yields the claim.
\end{proof}

The next example shows that, without uniform rigidity, the regional proximal relation of the non-autonomous system need not be the kernel of the maximal equicontinuous factor.

\begin{example}\label{ex:RP-not-general-MEF-kernel}
Consider the system from Example~\ref{ex:CC-no-UR-counterexample}. As shown there, $RP_{f_{1,\infty}}=X\times X.$

On the other hand, for $n\ge1$, $f_1^n(x)=R_\alpha^{\,n-1}(x_0)$  for all $x\in X$. Thus the family $\{f_1^n:n\ge0\}$ is equicontinuous. Hence $S_{eq}(f_{1,\infty})=\Delta_X.$

Consequently $RP_{f_{1,\infty}}$ is not, in general, the kernel of the maximal equicontinuous factor.
\end{example}

\begin{remark}\label{rem:classical-minimal-case}
Corollary~\ref{cor:RP-kernel-version} applies in particular whenever the autonomous limit system lies in a classical setting where the equicontinuous structure relation is the regional proximal relation. This includes, for example, the classical minimal abelian transformation-group framework of Veech and Ellis--Keynes \cite{EllisKeynes1971,Veech1968}.
\end{remark}

The key hypothesis in Theorem~\ref{thm:kernel-classification} is the preservation of $S_g$ by every generator. The next proposition collects useful sufficient criteria for this property.

\begin{proposition}\label{prop:descent-criteria}
Let $g\colon X\to X$ be continuous, let $\pi_g\colon (X,g)\to (X_g^{eq},T)$ be the autonomous maximal equicontinuous factor of $g$, and put $S_g=\ker\pi_g$. Let $(f_n)_{n\ge1}$ be a sequence of continuous self-maps of $X$. Each of the following conditions implies that $S_g$ is preserved by every generator $f_n$:
\begin{enumerate}
\item $f_n\circ g=g\circ f_n$ for every $n\ge 1$;
\item $\pi_g\circ f_n\circ g=\pi_g\circ g\circ f_n$ for every $n\in\mathbb N$;
\item each $f_n$ belongs to the enveloping semigroup $E(X,g)$.
\end{enumerate}
\end{proposition}

\begin{proof}
Clearly (1) implies (2).

Now suppose that condition~(2) holds. For each $n\ge1$, put $q_n=\pi_g\circ f_n$. Then $q_n\circ g=T\circ q_n.$ The image $Z_n=q_n(X)$ is a compact $T$-invariant subsystem of the equicontinuous system $(X_g^{eq},T)$, hence $(Z_n,T|_{Z_n})$ is equicontinuous. The map $q_n\colon (X,g)\to (Z_n,T|_{Z_n})$ is a continuous surjective factor map. By maximality of $\pi_g$, its kernel contains $S_g$. Therefore, if $\pi_g(x)=\pi_g(y)$, then
\[
 \pi_g(f_n(x))=q_n(x)=q_n(y)=\pi_g(f_n(y)),
\]
which means that $(f_n\times f_n)(S_g)\subseteq S_g$.

Now suppose that condition~(3) holds. Fix $n\ge 1$ and $(x,y)\in S_g$. Since $f_n\in E(X,g)$, there exists a net $(g^{m_\alpha})$ such that $g^{m_\alpha}(z)\longrightarrow f_n(z)$ for every $z\in X$.

For every $\alpha$, $(g^{m_\alpha}\times g^{m_\alpha})(S_g)\subseteq S_g$, because $\pi_g(g^{m_\alpha}(x))=T^{m_\alpha}(\pi_g(x))$. Since $S_g$ is closed, passing to the pointwise limit at the pair $(x,y)$ gives $(f_n(x),f_n(y))\in S_g$. Thus $S_g$ is preserved by every generator.
\end{proof}

\section*{Acknowledgments} 
The first author is supported by Junior Research Fellowship, E-certificate No. 24J/01/00293 and CSIR-HRDG Ref. No: June-24(ii)/EU-V, for carrying out doctoral work. The second and third authors are supported by Faculty Research Programme Grant, Ref. No./IoE/2025-26/12/FRP, Institution of Eminence, University of Delhi, India.

\bibliographystyle{plain}
\bibliography{references_MEF}

\end{document}